\documentclass[12pt]{article}
\usepackage{amssymb, amsmath}
\newcommand{\qed}{\hfill \rule{2.5mm}{2.5mm}}

\newcommand{\N}{{\mathbb N}}

\begin{document}
\newtheorem{thm}{Theorem}[section]
\newtheorem{defs}[thm]{Definition}
\newtheorem{lem}[thm]{Lemma}
\newtheorem{cor}[thm]{Corollary}
\newtheorem{prop}[thm]{Proposition}
\renewcommand{\theequation}{\arabic{section}.\arabic{equation}}
\newcommand{\newsection}[1]{\setcounter{equation}{0} \section{#1}}
\title{Strong sequential completeness of the natural domain of a conditional expectation operator in Riesz spaces
      }
\author{Wen-Chi Kuo\footnote{Supported in part by  NRF grant number CSUR160503163733.},
David F. Rodda \& Bruce A. Watson \footnote{Supported in part by the Centre for Applicable Analysis and
Number Theory and by NRF grant IFR170214222646 with grant no. 109289.} \\
School of Mathematics\\
University of the Witwatersrand\\
Private Bag 3, P O WITS 2050, South Africa }
\maketitle
\abstract{
Strong convergence and convergence in probability were generalized to the setting of a Riesz space with conditional expectation operator, $T$, in [{{\sc Y. Azouzi, W.-C. Kuo, K. Ramdane, B. A. Watson}, {Convergence in Riesz spaces with conditional expectation operators}, {\em Positivity}, {\bf 19} {(2015), 647-657}}] as $T$-strong convergence and convergence in $T$-conditional probability, respectively. Generalized $L^{p}$ spaces for the cases of $p=1,2,\infty$,  were discussed in the setting of Riesz spaces as $\mathcal{L}^{p}(T)$ spaces in [{{\sc C. C. A. Labuschagne, B. A. Watson}, {Discrete stochastic integration in Riesz spaces}, {\em Positivity}, {\bf 14} {(2010), 859-875}}].
An $R(T)$ valued norm, for the cases of $p=1,\infty,$ was introduced on these spaces in 
[{{\sc W. Kuo, M. Rogans, B.A. Watson}, {Mixing processes in Riesz spaces}, {\em Journal of Mathematical Analysis and Application}, {\bf 456} {(2017), 992-1004}}]
where it was also shown that $R(T)$ is a universally complete $f$-algebra and that these spaces are $R(T)$-modules.
In [{{\sc Y. Azouzi, M. Trabelsi}, {$L^p$-spaces with respect to conditional expectation on Riesz spaces}, {\em Journal of Mathematical Analysis and Application}, {\bf 447} {(2017), 798-816}}] functional calculus was used to consider  $\mathcal{L}^{p}(T)$  for $p\in (1,\infty)$.
In this paper we prove the strong sequential completeness of the space $\mathcal{L}^{1}(T)$, the natural domain of the conditional expectation operator $T$,
and the strong completeness of  $\mathcal{L}^{\infty}(T)$. 
}
\parindent=0in
\parskip=.2in

\section{Introduction}

Strong convergence and convergence in probability were generalized to Dedekind complete Riesz spaces with a conditional expectation operator in \cite{AKRW} as $T$-strong convergence and $T$-conditional convergence in conditional probability, respectively. Generalized $L^{p}$ spaces for $p=1,2,\infty$ were discussed in the setting of Riesz spaces as $\mathcal{L}^{p}(T)$ spaces in \cite{LW}. An $R(T)$ valued norm, for the cases of $p=1,\infty,$ was introduced on the $\mathcal{L}^{p}(T)$  spaces in \cite{KRW} where it was also shown that $R(T)$ is a universally complete $f$-algebra and that these spaces are $R(T)$-modules.
More recently, in \cite{AT}, the $\mathcal{L}^{p}(T)$, for $p\in (1,\infty)$, spaces were considered. We also refer the reader to \cite{JHVDW} for an interesting study of sequential order convergence in 
vector lattices using convergence structures and filters, and to \cite{Berberian} for the well known proof of the strong sequential completeness of $L^1(\Omega,\mathcal{F},\mu)$.

In this paper we prove the strong sequential completeness of the natural domain,  $\mathcal{L}^{1}(T)$, of the Riesz space conditional expectation operator $T$, i.e. that each strong Cauchy sequence in $\mathcal{L}^{1}(T)$ converges strongly in $\mathcal{L}^{1}(T)$.
The term strong here means  with respect to the vector valued norm induced by the conditional expectation operator $T$ in the given space. These results can be extended to the convergence of strong Cauchy nets which contain a sequence as a subnet. We conclude by showing the strong completeness of $\mathcal{L}^{\infty}(T)$, i.e. that every strong Cauchy net in $\mathcal{L}^{\infty}(T)$ is strongly convergent. 

\section{Preliminaries}

A conditional expectation operator, $T$, on a Dedekind complete Riesz space, $E$, with weak order unit, say $e$,
is a positive order continuous projection which maps weak order units to weak order units and has $R(T)$ a Dedekind complete Riesz subspace of $E$, see \cite{KLW2}. In addition we assume in this paper that $T$ is strictly positive, in that if $v\in E_+$ with $v\ne 0$ then  $Tv\ne 0$ ($Tv\ge 0$ as $T$ is positive). This last condition is required for both the construction of the $T$-universal completion of $E$, i.e. the natural domain,  $\mathcal{L}^{1}(T)$, of $T$ and so that the mapping $f\mapsto T|f|$ defines an $R(T)$ valued norm on  $\mathcal{L}^{1}(T)$.

The Riesz space 
$\mathcal{L}^1(T)$ is defined to be the $T$-universal completion of $E$ or natural domain of $T$, see \cite{dePagteretal} and \cite{KLW2}. We recall that $T$ has a unique extension to $\mathcal{L}^1(T)$ as a conditional expectation operator. In particular $\mathcal{L}^1(T)$ is characterized by the property that if $(x_\alpha)$ is an upward directed net in $\mathcal{L}^1(T)$ with $(Tx_\alpha)$ bounded in $E^u$ (the universal completion), then $(Tx_\alpha)$ is order convergent in $\mathcal{L}^1(T)$.

We recall from \cite{KRW} that in $\mathcal{L}^1(T)$, $R(T)$ is a universally complete $f$-algebra and that  $\mathcal{L}^1(T)$ is an  $R(T)$-module. It thus makes sense, as was done in   \cite{KRW}, to define an $R(T)$-valued norm on $\mathcal{L}^1(T)$ by
$\|f\|_{T,1}:=T|f|$. This norm takes its values in $R(T)^+$, is homogeneous with respect to multiplication by elements of $R(T)^+$, is strictly positive and obeys the triangle inequality. For more details on this norm we refer the reader to \cite{KRW}. Convergence with respect to this norm was called $T$-strongly convergence in \cite{AKRW} where various of its properties were studied in relation to other modes of convergence.  

The other space that will be of interest in this work is 
$$\mathcal{L}^\infty(T):=\{f\in \mathcal{L}^1(T)\,:\,|f|\le g \mbox{ for some } g\in R(T)\}$$
with $R(T)$-valued norm
$$\|f\|_{T,\infty}:=\inf\{g\in R(T)\,:\, |f|\le g\}.$$
 We refer the reader to  \cite{KRW} for more details.
 
 The following lemma will assist in the proof of strong sequential completeness.
 
\begin{lem}\label{Lem1}
 Let $(h_{n})$ be a sequence in $\mathcal{L}^{1}(T)$ with 
$\begin{displaystyle}s:=\sum_{n=1}^{\infty}T|h_{n}|\end{displaystyle}$ 
order convergent in the universal completion of $\mathcal{L}^{1}(T)$, then the summation 
$\begin{displaystyle}\sum_{n=1}^{\infty}h_{n}\end{displaystyle}$ 
is order convergent in $\mathcal{L}^{1}(T)$.
\end{lem}

\begin{proof}
Let $\begin{displaystyle}s_n^\pm=\sum_{i=1}^n h_{i}^\pm\end{displaystyle}$, then the partial sums $s_n$ of $\begin{displaystyle}\sum_{n=1}^{\infty}h_{n}\end{displaystyle}$ are given by $s_n=s_n^+-s_n^-$. Here $(s_n^\pm)$ are  increasing sequences with 
$$Ts_n^\pm=\sum_{i=1}^n Th_{i}^\pm\le\sum_{i=1}^n T|h_{i}|\le s.$$
The $T$-universal completeness of $\mathcal{L}^{1}(T)$ now allows us to conclude
that $(s_n^\pm)$ are convergent in  $\mathcal{L}^{1}(T)$  to limits, say $h^\pm$.
Setting $h=h^+-h^-$ we have that
$$s_n=s_n^+-s_n^-\to h^+-h^-=h\in \mathcal{L}^{1}(T)$$
in order as $n\to \infty$.\qed
\end{proof}

\begin{defs}
	We say that a net  $(f_\alpha)$ in ${\mathcal{L}}^p(T)$, $p=1,\infty$, is a strong Cauchy net if 
	$$v_\alpha:=\sup_{\beta,\gamma\ge \alpha}\|f_\beta-f_\gamma\|_{T,p}$$
	is eventually defined and has order limit zero.
\end{defs}
\section{Strong sequential completeness of ${\mathcal{L}}^1(T)$}
We now show that ${\mathcal{L}}^1(T)$ is strongly sequentially complete - i.e. that for 
every sequence $(f_n)$ in ${\mathcal{L}}^1(T)$ with 
$\sup_{i,j\ge n}T|f_i-f _j| \downarrow 0$ there is $f\in {\mathcal{L}}^1(T)$
so that $T|f_n-f|\to 0$ in order as $n\to\infty$.

\begin{thm}
	$\mathcal{L}^{1}(T)$ is strongly sequentially complete. 
\end{thm}

\begin{proof}
Let $(f_n)$ be a strong $T$-Cauchy sequence in $\mathcal{L}^{1}(T)$.
From the definition of a strong Cauchy sequence, we can define
 $$v_n:=\sup_{r,s\geq n}T|f_r-f_s|$$ where the sequence $(v_n)\subset R(T)$ decreases with infimum zero. 
As $e, v_n\in R(T)$, it follows that $(\frac{1}{2^{j}}e-v_{n})^{+}\in R(T)$ and hence
the band projections $P_{j,n}:=P_{(\frac{1}{2^{j}}e-v_{n})^{+}}, j,n\in\mathbb{N},$
commute with $T$, see \cite{KLW2}.
For $n=0$ define $P_{j,n}=0$. 
We observe that $P_{j,n}$ is increasing in $n$ and decreasing in $j$. 
In particular, $\begin{displaystyle}\lim_{n\rightarrow\infty}P_{j,n}=I\end{displaystyle}$, since $v_n\downarrow 0$.
Hence $\begin{displaystyle}\sum_{n=0}^{\infty}(P_{j,n+1}-P_{j,n})=I\end{displaystyle}$ for each $j\in\mathbb{N}$.

We now construct a sequence $(g_j)\in \mathcal{L}^{1}(T)$ that is both asymptotically close to $(f_n)$ and is strongly convergent in $\mathcal{L}^{1}(T)$. 	
As band projections commute with Riesz space absolute value, we have
$$T|(P_{j,n}-P_{j,n-1})f_{\max\{j,n\}}|= (P_{j,n}-P_{j,n-1})T|f_{\max\{j,n\}}|,\quad n,j\in \N.$$ 
Here, for $m\ne n$, $(P_{j,n}-P_{j,n-1})\wedge
(P_{j,m}-P_{j,m-1})=0$
so  
\begin{align*}
\sum_{n=1}^{\infty}T|(P_{j,n}-P_{j,n-1})f_{\max\{j,n\}}|
&=\sum_{n=1}^{\infty}(P_{j,n}-P_{j,n-1})T|f_{\max\{j,n\}}|\\
&=\sup_{n\in\N}(P_{j,n}-P_{j,n-1})T|f_{\max\{j,n\}}|\\
&=:K\in E^u
\end{align*}
exists in the universal completion $E^u$.
 Lemma \ref{Lem1}
can now be applied to give that the summation
$$g_j=\sum_{n=1}^{\infty}(P_{j,n}-P_{j,n-1})f_{\max\{j,n\}},\quad j\in \N,$$ 
converges in order in $\mathcal{L}^{1}(T)$.

We now show that the sequence $(g_j)$ converges in $\mathcal{L}^{1}(T)$.
Consider $T|g_{j}-g_{j+1}|$. Because $\begin{displaystyle}\sum_{n=0}^{\infty}(P_{j,n+1}-P_{j,n})=I\end{displaystyle}$ for each $j\in\mathbb{N}$, we have that 
\begin{eqnarray*}
	T|g_{j}-g_{j+1}|& =& T\left|\sum_{n=1}^{\infty}(P_{j,n}-P_{j,n-1})f_{\max\{j,n\}}-\sum_{m=1}^{\infty}(P_{j+1,m}-P_{j+1,m-1})f_{\max\{j+1,m\}}\right|\\
	& = &
	T\left|\sum_{n=1}^{\infty}\sum_{m=1}^{\infty}(P_{j+1,m}-P_{j+1,m-1})(P_{j,n}-P_{j,n-1})(f_{\max\{j,n\}}-f_{\max\{j+1,m\}})\right|\\
	& = &
	\sum_{n=1}^{\infty}\sum_{m=1}^{\infty}(P_{j+1,m}-P_{j+1,m-1})(P_{j,n}-P_{j,n-1})T|f_{\max\{j,n\}}-f_{\max\{j+1,m\}}|.
	\end{eqnarray*}
Here we have used that $$(P_{j+1,m}-P_{j+1,m-1})(P_{j,n}-P_{j,n-1})\wedge
(P_{j+1,x}-P_{j+1,x-1})(P_{j,y}-P_{j,y-1})=0$$ for $(m,n)\ne (x,y)$.

For $m\ge n$ we have
\begin{align*}
(P_{j,n}-P_{j,n-1})T|f_{\max\{j,n\}}-f_{\max\{j+1,m\}}|
&\le (P_{j,n}-P_{j,n-1})v_n\\ &\le \frac{1}{2^j}(P_{j,n}-P_{j,n-1})e
\end{align*}
while for $m< n$ we have
\begin{align*}
(P_{j+1,m}-P_{j+1,m-1})T|f_{\max\{j,n\}}-f_{\max\{j+1,m\}}|
&\le (P_{j+1,m}-P_{j+1,m-1})v_m\\
&\le \frac{1}{2^{j+1}}(P_{j+1,m}-P_{j+1,m-1})e.
\end{align*}
Thus
\begin{eqnarray*}
	T|g_{j}-g_{j+1}|
	\le 
	\frac{1}{2^j}\sum_{n=1}^{\infty}\sum_{m=1}^{\infty}(P_{j+1,m}-P_{j+1,m-1})(P_{j,n}-P_{j,n-1})e
	=\frac{1}{2^j}e
	\end{eqnarray*}
and the summation
$\begin{displaystyle}\sum_{j=1}^\infty T|g_{j}-g_{j+1}|\end{displaystyle}$
is $e$-uniformly (and hence order) convergent. Application of Lemma \ref{Lem1} gives that
the summation 
$\begin{displaystyle}\sum_{j=1}^\infty (g_{j}-g_{j+1})\end{displaystyle}$
is order convergent, which is equivalent to the order limit
$\begin{displaystyle}\lim_{j\to\infty}(g_1-g_{j+1})\end{displaystyle}$
existing. We can thus define $g$ to be the order limit of the sequence $(g_j)$ in $\mathcal{L}^1(T)$. Order continuity of $T$ now gives that $\lim_{n\to\infty} T|g_n-g|=0$ and that $(g_n)$ converges strongly to $g$ in $\mathcal{L}^1(T)$.

From the order continuity of $T$ and the order convergence of $(g_n)$ to $g$ we have that $T|g_n-g|\to 0$ in order. Hence to show 
that $g$ is the strong limit of the $(f_n)$
it suffices prove that $T|g_n-f_n|\rightarrow 0$ in order. 
As
$\begin{displaystyle}\sum_{n=0}^{\infty}(P_{j,n+1}-P_{j,n})=I\end{displaystyle}$ we have
\begin{eqnarray*}
 g_j-f_j &=& \sum_{n=1}^{\infty}(P_{j,n}-P_{j,n-1})(f_{\max\{j,n\}}-f_j)\\
&=&\sum_{n=1}^{j}(P_{j,n}-P_{j,n-1})(f_j-f_j)
+\sum_{n=j+1}^{\infty}(P_{j,n}-P_{j,n-1})(f_n-f_j)\\
&=& \sum_{n=j+1}^{\infty}(P_{j,n}-P_{j,n-1})(f_n-f_j).
\end{eqnarray*}
The order continuity of $T$ gives
\begin{eqnarray*}
T|g_j-f_j| 
&\le& \sum_{n=j+1}^{\infty}(P_{j,n}-P_{j,n-1})T|f_n-f_j|\\
&\le& \sum_{n=j+1}^{\infty}(P_{j,n}-P_{j,n-1})v_j\\
&\le& v_j
\end{eqnarray*}
and $v_j\downarrow 0$ as $j\to\infty$. Thus $T|f_j-g|\to 0$ in order as $n\to\infty$.
\qed
\end{proof}

These results extended to the convergence of strong Cauchy nets which contain a sequence as a subnet.
More can be said in the case of $p=\infty$, as see in the following section.
\section{Strong completeness of $\mathcal{L}^\infty(T)$}
For the case of $\mathcal{L}^\infty(T)$ we can prove a stronger result, i.e. that $\mathcal{L}^\infty(T)$ is strongly complete. The proof, unfortunately, cannot be extended to 
$\mathcal{L}^p(T)$ for $p\in [1,\infty)$.

\begin{thm}
Each strong Cauchy net in  $\mathcal{L}^\infty(T)$  is strongly convergent in  $\mathcal{L}^\infty(T)$.
\end{thm}

\begin{proof}
Let $(f_\alpha)$ be a  strong Cauchy net in  $\mathcal{L}^\infty(T)$, then
eventually 
\begin{align*}
v_\alpha:=\sup_{\beta,\gamma\ge\alpha} \|f_\beta-f_\gamma\|_{T,\infty}
=\inf\{g\in R(T)\,:\,|f_\beta-f_\gamma|\le g \mbox{ for all } \beta,\gamma\ge\alpha\}
\end{align*}  
exists as an element of $R(T)$ and $v_\alpha\downarrow 0$. Hence eventually $|f_\beta-f_\gamma|\le v_\alpha$, for $\beta,\gamma\ge \alpha$, and the Cauchy net $(f_\alpha)$ is eventually bounded. We can thus define
$\underline{f}:=\liminf_\alpha f_\alpha$, $\overline{f}:=\limsup_\alpha f_\alpha$ in
 $\mathcal{L}^\infty(T)$.
Now 
$$0\le\overline{f}-\underline{f}
=\lim_\alpha(\sup_{\beta\ge\alpha}f_\beta-\inf_{\gamma\ge\alpha}f_\gamma)
=\lim_\alpha\sup_{\beta,\gamma\ge\alpha}(f_\beta-f_\gamma)\le \lim_\alpha v_\alpha=0.$$
So $\overline{f}=\underline{f}$ and we can set $f:=\overline{f}=\underline{f}$ with $(f_\alpha)$ converging in order $f$, see \cite{AKRW}.
However
$|f_\beta-f_\gamma|\le v_\alpha$ for all $\beta,\gamma\ge\alpha$, so taking the order limit in the index $\gamma$ we have 
$|f_\alpha-f|\le  v_\alpha$ and hence
$\|f_\alpha-f\|_{T,\infty}\le  v_\alpha\downarrow 0$.
\qed
\end{proof}


\begin{thebibliography}{99}

\bibitem{AbrSir}{{\sc Y. Abramovich, G. Sirotkin}, 
{On Order Convergence of Nets}, {\em Positivity}, {\bf 9} {(2005), 287-292}.}

\bibitem{AKRW}{{\sc Y. Azouzi, W.-C. Kuo, K. Ramdane, B. A. Watson}, 
{Convergence in Riesz spaces with conditional expectation operators}, 
{\em Positivity}, {\bf 19} {(2015), 647-657.}}	

\bibitem{AT}{{\sc Y. Azouzi, M. Trabelsi}, {$L^p$-spaces with respect to conditional expectation on Riesz spaces}, {\em J. Math. Anal. Appl.}, {\bf 447} {(2017), 798-816}}

\bibitem{Berberian}{{\sc S.K. Berberian}, 
{\em Measure and Integration}, {AMS Chelsea, 1965.}}


\bibitem{deJonge}{{\sc E. de Jonge}, 
{Conditional expectation and ordering}, 
{\em Ann. Probab.}, {\bf 7} {(1979), 179-183.}}

\bibitem{dePagteretal}{{\sc B. de Pagter, P. G. Dodds, C. B. Huijsmans}, 
{Characterizations of conditional expectation-type operators}, 
{\em Pacific J. Math.}, {\bf 141} {(1990), 55-77.}}

\bibitem{GL}{\sc J.J. Grobler, C.C.A. Labuschagne,}
{The It\^o integral for Brownian motion in vector lattices: Part 2.}
{\em J. Math. Anal. Appl.} {\bf 423} {(2015), 820-833}.

\bibitem{KLW2}{{\sc W.-C. Kuo, C. C. A. Labuschagne, B. A. Watson}, 
{Conditional expectations on Riesz spaces}, 
{\em J. Math. Anal. Appl.}, {\bf 303} {(2005), 509-521.}}

\bibitem{KRW}{{\sc W. Kuo, M. Rogans, B.A. Watson}, 
{Mixing processes in Riesz spaces}, 
{\em J. Math. Anal. Appl.}, {\bf 456} {(2017), 992-1004.}}

\bibitem{LW}{{\sc C. C. A. Labuschagne, B. A. Watson}, 
{Discrete stochastic integration in Riesz spaces}, 
{\em Positivity}, {\bf 14} {(2010), 859-875.}}	

\bibitem{Lux&Zaan}{{\sc W. A. J. Luxemburg, A. C. Zaanen}, 
{\em Riesz Spaces}, 
{I, North Holland, Amsterdam, 1971.}}

\bibitem{Stoica}{{\sc G. Stoica}, 
{Martingales in vector lattices}, 
{\em Bull. Math. Soc. Sci. Math. Roumanie, N.S.}, {\bf 34 (82)} {(1990), 357-362.}}

\bibitem{Troitsky}{{\sc V. Troitsky}, 
{Martingales in Banach lattices}, 
{\em Positivity}, {\bf 9} {(2005), 437-456.}}	

\bibitem{JHVDW}{{\sc J. H. van der Walt}, 
{The order convergence structure}, 
{\em Indagationes Mathematicae}, {\bf 21} {(2011), 138-155.}}	

\bibitem{Zaan}{{\sc A. C. Zaanen}, 
{\em Introduction to Operator Theory in Riesz Spaces}, {Springer-Verlag, 1997.}}

\end{thebibliography}
\end{document}